\documentclass[11pt,letterpaper]{article}
\usepackage[T1]{fontenc}
\usepackage[margin=1.05in]{geometry}
\usepackage{amsmath,amssymb,amsthm}
\usepackage{microtype}
\usepackage{xcolor}
\usepackage{enumitem}
\usepackage{fancyhdr}
\usepackage{hyperref}
\definecolor{paperlink}{RGB}{0,0,114}
\hypersetup{colorlinks=true,linkcolor=paperlink,citecolor=paperlink,urlcolor=paperlink,
  pdftitle={Bounded Toeplitz operators and unbounded projected multiplication},pdfauthor={Sam Looi},
  pdfcreationdate={D:20260224000000+00'00'},pdfmoddate={D:20260224000000+00'00'}}
\setlist[enumerate]{leftmargin=2em,itemsep=0.25em,topsep=0.4em}
\theoremstyle{plain}
\newtheorem{theorem}{Theorem}[section]
\newtheorem{proposition}[theorem]{Proposition}
\newtheorem{lemma}[theorem]{Lemma}
\newtheorem{corollary}[theorem]{Corollary}
\theoremstyle{definition}
\newtheorem{definition}[theorem]{Definition}
\theoremstyle{remark}
\newtheorem{remark}[theorem]{Remark}
\numberwithin{equation}{section}
\newcommand{\C}{\mathbb C}
\newcommand{\dd}{\,\mathrm{d}}
\newcommand{\norm}[1]{\left\|#1\right\|}
\newcommand{\Ima}{\operatorname{Im}}
\DeclareMathOperator{\tr}{tr}
\title{Bounded Toeplitz operators and unbounded projected multiplication}
\author{Sam Looi\thanks{Division of Physics, Mathematics and Astronomy, California Institute of Technology. Email: \href{mailto:looi@caltech.edu}{looi@caltech.edu}.}}
\date{July 20, 2026}
\begin{document}
\maketitle

\begin{abstract}
We construct a smooth, nonnegative radial symbol $g$ on $\C$ whose Toeplitz operator on Bargmann--Fock space is positive with singular values decaying faster than every exponential, while projected multiplication $f\mapsto P(gf)$ has no bounded extension from the ambient Gaussian $L^2$ space. All heat transforms of $g$ at positive times are Schwartz functions, and $g\in L^p(\dd A)$ for every $0<p<2$. We also construct smooth real symbols realizing every possible set of heat boundedness times. Any prescribed positive threshold may be included or excluded, with Schwartz transforms at every later time. All these symbols satisfy the coherent state domain condition, and all heat integrals converge absolutely.
\end{abstract}

\section{Introduction}

For nonnegative symbols, classical Fock--Carleson theory identifies boundedness of the Toeplitz compression with boundedness of the heat transform at one, and hence every, positive time~\cite[Theorems~3.1 and~4.2]{HuLv2011}. We prove that these equivalent conditions do not imply boundedness of projected multiplication on the ambient Gaussian $L^2$ space, even when every heat transform at a positive time is a Schwartz function and the singular values of the compression decay faster than every exponential. The symbol is smooth, nonnegative and radial. Its Gaussian averages have a uniform bound, which makes the compression bounded, but the averages of its square have no such bound, so projected multiplication is unbounded. For signed symbols, we construct examples for every possible upper interval of heat boundedness times, with all heat integrals absolutely convergent. Both constructions satisfy the coherent state domain condition.

We work with the Gaussian measure and reproducing kernel
\begin{equation}\label{eq:norm}
\dd\mu(z)=\pi^{-1}e^{-|z|^2}\dd A(z),\qquad K(z,w)=e^{z\overline w}\,,
\end{equation}
where $\dd A$ is planar Lebesgue measure. The Bargmann--Fock space is
\begin{equation*}
F^2=F^2(\C)=\left\{f\text{ entire}:\norm f_{F^2}^2
                         =\int_{\C}|f|^2\dd\mu<\infty\right\}\,.
\end{equation*}
Let $P:L^2(\mu)\to F^2$ be the orthogonal projection, and take inner products linear in the first variable. Symbols are measurable and finite almost everywhere.

For a symbol $g$, multiplication on $L^2(\mu)$ has domain
\begin{equation}\label{eq:Ug-def}
\mathcal D(M_g)=\{f\in L^2(\mu):gf\in L^2(\mu)\},\qquad
U_gf=P(gf)\quad(f\in\mathcal D(M_g))\,.
\end{equation}
For $f\in L^2(\mu)$, the truncations $f\mathbf1_{\{|g|\le N\}}$ lie in $\mathcal D(M_g)$ and converge to $f$ in $L^2(\mu)$, which proves density. We shall call $U_g$ bounded when it has a bounded extension from $L^2(\mu)$ to $F^2$. The formula in~\eqref{eq:Ug-def} also defines the usual Toeplitz compression when restricted to $F^2\cap\mathcal D(M_g)$; for $U_g$ we allow every $f\in\mathcal D(M_g)$.

To specify the Toeplitz realization, put
\begin{equation*}
\mathcal D(\mathfrak t_g)=\{f\in F^2:|g|^{1/2}f\in L^2(\mu)\},\qquad
\mathfrak t_g(f,h)=\int_{\C}g f\overline h\dd\mu\,.
\end{equation*}
The integral is absolutely convergent for $f,h\in\mathcal D(\mathfrak t_g)$. When this domain is dense in $F^2$ and the form has a bounded extension to $F^2\times F^2$, we denote the unique representing operator by $T_g$. In that case, $T_gf=U_gf$ for $f\in F^2\cap\mathcal D(M_g)$. If $g\ge0$, we may equivalently require the integral itself to define a bounded form on all of $F^2\times F^2$ (Theorem~\ref{thm:two-kernel-tests}).

The normalized reproducing kernels, or coherent states, are $k_a(z)=e^{z\overline a-|a|^2/2}$. We consider the domain condition
\begin{equation}\label{eq:BC-domain}
gk_a\in L^2(\mu)\qquad(a\in\C)\,.
\end{equation}
Their span is dense in $F^2$, so~\eqref{eq:BC-domain} makes both the Toeplitz compression and $\mathfrak t_g$ densely defined. The identity
\begin{equation}\label{eq:BC-window-pointwise}
\norm{gk_a}_2^2=\frac1\pi\int_{\C}e^{-|z-a|^2}|g(z)|^2\dd A(z)
\end{equation}
shows that this is a pointwise finiteness condition on Gaussian averages. By contrast, boundedness of $U_g$ requires a uniform bound over $a$.

For $t>0$, we use the heat transform
\begin{equation}\label{eq:heat-definition}
g^{(t)}(a)=\frac1{4\pi t}\int_{\C}e^{-|z-a|^2/(4t)}g(z)\dd A(z)\,,
\end{equation}
whenever the integral is absolutely convergent. For nonnegative symbols we also allow the value $+\infty$ when discussing finiteness. Every heat transform that we assert is bounded is given by an absolutely convergent integral. Under~\eqref{eq:BC-domain}, the Berezin transform of a bounded $T_g$ is $g^{(1/4)}$ in this normalization. We write $\mathcal S(\C)$ for the Schwartz space on $\C\simeq\mathbb R^2$. For $q>0$, the Schatten class $S_q$ consists of the compact operators on $F^2$ whose singular values $s_j$, listed in decreasing order, satisfy $\sum_{j\ge1}s_j^q<\infty$.

\begin{theorem}\label{thm:main}
There is a smooth, nonnegative radial symbol $g$ on $\C$ such that $gk_a\in L^2(\mu)$ for every $a\in\C$ and $g\in L^p(\dd A)$ for every $0<p<2$. The Toeplitz operator $T_g:F^2\to F^2$ is positive and bounded, but $U_g:\mathcal D(M_g)\subset L^2(\mu)\to F^2$ has no bounded extension.

For this same symbol, $g^{(t)}\in\mathcal S(\C)$ for every $t>0$, and for every $c>0$ there is a constant $C_c<\infty$ such that
\begin{equation}\label{eq:singular-decay}
s_j(T_g)\le C_c e^{-cj}\qquad(j\ge1)\,.
\end{equation}
In particular, $T_g\in S_q$ for every $q>0$.
\end{theorem}

The integrability range cannot include $p=2$: for $g\in L^2(\dd A)$, the right side of~\eqref{eq:BC-window-pointwise} is uniformly bounded in $a$, so Theorem~\ref{thm:carleson} gives boundedness of $U_g$. We prove the singular value bound by diagonalizing the radial Toeplitz operator and using the integrability of $g(z)e^{A|z|^2}$ for every $A>0$. Proposition~\ref{prop:Tg} gives the trace formula for nonnegative symbols.

\begin{corollary}\label{cor:global}
For the symbol in Theorem~\ref{thm:main}, every integral in~\eqref{eq:BC-window-pointwise} is finite, but
\begin{equation*}
\sup_{a\in\C}\int_{\C}e^{-|z-a|^2}|g(z)|^2\dd A(z)=\infty\,.
\end{equation*}
There is a sequence $a_n\in\C$ with $|a_n|\to\infty$ along which the integrals tend to infinity.
\end{corollary}

We build $g$ from nonnegative radial bumps on thin annuli, so that each bump occupies a full annulus and no cancellation is involved. We choose their heights and widths so that the integrals of the bumps are summable, although the integral of $g^2$ over a small sector of each annulus tends to infinity. Condition~\eqref{eq:BC-domain} makes $g$ locally square integrable, so its restriction to any compact set gives a bounded projected multiplication map. The unboundedness must therefore come from arbitrarily large radii. In Section~\ref{sec:carleson} we give this argument and compare it with the obstruction caused by local singularities of power symbols.

In Proposition~\ref{prop:polynomial-growth}, we construct a nonnegative symbol of polynomial growth for which $T_g$ is positive and trace class, $U_g$ is unbounded, and every heat transform at a positive time lies in $C_0(\C)$. These heat transforms cannot all be Schwartz functions: for a nonnegative symbol of polynomial growth, that would force $U_g$ to be bounded, as we show at the end of Section~\ref{sec:polynomial-growth}.

\paragraph{Boundedness criteria and local obstructions.}
We obtain the characterization of $U_g$ in Theorem~\ref{thm:carleson} by combining an adjoint identity with the classical Fock--Carleson criterion. Boundedness of $U_g$ is equivalent to boundedness of multiplication by $g$ from $F^2$ to $L^2(\mu)$, to the Fock--Carleson condition for $|g|^2\dd\mu$, or to boundedness of the integral form $\mathfrak t_{|g|^2}$ on $F^2\times F^2$. In terms of coherent states, the criterion is
\begin{equation*}
U_g\text{ has a bounded extension }L^2(\mu)\to F^2
\quad\Longleftrightarrow\quad
\sup_{a\in\C}\norm{gk_a}_2^2<\infty\,.
\end{equation*}
When these conditions hold, $U_g^*h=\overline g\,h$ and $\norm{U_g}^2=\norm{T_{|g|^2}}$. For $g\ge0$, boundedness of $T_g$ is equivalent to a uniform bound on the linear Gaussian averages $g^{(1/4)}$, by Theorem~\ref{thm:two-kernel-tests}. The symbol in Theorem~\ref{thm:main} satisfies this bound, although its quadratic averages become arbitrarily large as the center varies; each individual quadratic average is still finite.

For $g(z)=|z|^\alpha$, Proposition~\ref{prop:powers} gives the range \mbox{$-2<\alpha\le0$} for boundedness of both $T_g$ and every heat transform at a positive time, and the smaller range \mbox{$-1<\alpha\le0$} for boundedness of $U_g$. In the separation range \mbox{$-2<\alpha\le-1$}, the symbol fails~\eqref{eq:BC-domain} because of its singularity at the origin. More generally, Remark~\ref{rem:cgt} shows that $F^2\subset\mathcal D(M_g)$ holds if and only if $U_g$ has a bounded extension. Thus the coherent state condition is weaker than requiring the multiplication domain to contain the whole Fock space.

\paragraph{Prescribing the times of heat boundedness.}
For nonnegative symbols, the classical positive Toeplitz criterion makes boundedness of $T_g$ equivalent to boundedness of the heat transform at one, or every, positive time; we record this in Corollary~\ref{cor:positive-heat-criterion}, using the quantitative estimate in Lemma~\ref{lem:positive-heat}. We next ask how heat boundedness can depend on time when the symbol changes sign. Absolute convergence of every heat integral and the semigroup property require the set of boundedness times to be an upper interval. In Theorem~\ref{thm:heat-irreversibility}, we construct a smooth real symbol satisfying~\eqref{eq:BC-domain} for every permitted set: besides the empty set and $(0,\infty)$, for any $t_0>0$ we realize either
\begin{equation*}
\bigl\{t>0:g^{(t)}\in L^\infty(\C)\bigr\}=[t_0,\infty)
\quad\text{or}\quad(t_0,\infty)\,.
\end{equation*}
For either endpoint choice, we obtain a Schwartz heat transform at every $t>t_0$, while every transform at $0<t<t_0$ is unbounded. The proper intervals therefore exhibit a dependence on heat time that positivity excludes.

\paragraph{Relation to heat flow criteria.}
Berger and Coburn related the boundedness of Toeplitz operators to heat transforms of their symbols~\cite{BergerCoburn1994HeatFlow}. Their endpoint conjecture, as recalled by Bauer and Fulsche~\cite[Section~3]{BauerFulsche2019}, proposed that boundedness of the Toeplitz compression, under~\eqref{eq:BC-domain}, is equivalent to boundedness of a heat transform at a critical time. In our normalization, the Berezin time is $1/4$ and this endpoint time is $1/8$. Under the normalization $\dd\mu_{BC}=(2\pi)^{-1}e^{-|z|^2/2}\dd A(z)$ used in~\cite{looi2026counterexample}, these times are $1/2$ and $1/4$, respectively: the change of variables $z=\sqrt2 w$ divides heat times by two.

In the companion paper~\cite{looi2026counterexample}, the author settled the Berger--Coburn boundedness conjecture and the Bauer--Coburn--Isralowitz compactness conjecture by disproving both implications of each in every complex dimension. The boundedness conjecture was formulated in~\cite{BergerCoburn1994HeatFlow} and restated as Conjecture~1 in~\cite{BCI2010}. The compactness conjecture, stated as Conjecture~2 in~\cite{BCI2010}, asserted that $T_g$ is compact if and only if the endpoint heat transform belongs to $C_0(\C^n)$. The author constructed symbols on $\C^n$ satisfying the coherent state condition for which $T_g$ is Hilbert--Schmidt, hence compact, but the endpoint heat transform is unbounded. The same work also gave examples satisfying this condition for which the endpoint heat transform belongs to $C_0(\C^n)$ but $T_g$ has no bounded extension to the Fock space. The failure of both implications of both conjectures, in every complex dimension, was also announced in the author's ICM~2026 poster abstract~\cite[p.~231]{LooiICM2026}.

The Berger--Coburn conjecture was already known to hold for nonnegative symbols~\cite[Proposition~3.6]{BauerFulsche2019}, for symbols in the Fock space $\mathrm{BMO}^1$ class, which contains classical $\mathrm{BMO}(\mathbb R^{2n})$~\cite[Theorem~6]{BCI2010}, and for the classes of metaplectic Toeplitz operators treated in~\cite{CHS2023,Xiong2023}. The present paper asks whether boundedness of the compression, together with boundedness of every heat transform at a positive time, implies boundedness of projected multiplication on the ambient Gaussian space.

In the examples of~\cite{looi2026counterexample} with an unbounded endpoint heat transform, the Toeplitz compression is bounded but $U_g$ is not. Indeed, boundedness of $U_g$ makes $\sup_a\norm{gk_a}_2$ finite by Theorem~\ref{thm:carleson}, and Cauchy--Schwarz gives $|g^{(1/8)}(a)|\le2\norm{gk_a}_2$. We prove in Theorem~\ref{thm:main} that $U_g$ can be unbounded for a nonnegative symbol even when every heat transform at a positive time is Schwartz and the singular values of $T_g$ satisfy~\eqref{eq:singular-decay}. Thus neither cancellation nor an unbounded heat transform is needed for the failure of ambient boundedness.

\paragraph{Related work.}
The embedding criterion used here belongs to the classical theory of Fock--Carleson measures and positive Toeplitz operators~\cite{IsralowitzZhu2010Toeplitz}; see also~\cite[Theorem~3.29]{ZhuFock}. Definitions using sesquilinear forms also occur in the work of Esmeral, Rozenblum, and Vasilevski~\cite{EsmeralRozenblumVasilevski2019LinvariantFockCarleson}, who study Fock--Carleson measures controlling derivatives and the associated Toeplitz operators. Nonnegative unbounded radial symbols with compact Toeplitz operators already appear in~\cite[Section~4, Example~4]{GrudskyVasilevski2002RadialComponentEffects}. We choose the heights and widths of the annular bumps so that the quadratic coherent state averages, although finite at each fixed center, tend to infinity along a sequence of centers, and $g$ still belongs to $L^p(\dd A)$ for every $0<p<2$.

Additional symbol regularity can give characterizations by the Berezin transform. Coburn, Isralowitz, and Li~\cite[Theorems~1--4]{CoburnIsralowitzLi2011} establish boundedness and compactness criteria for Gaussian $\mathrm{BMO}^1$ symbols on the Segal--Bargmann space. Bauer, Coburn, and Isralowitz~\cite{BCI2010} develop heat flow criteria for compactness. For these Gaussian $\mathrm{BMO}^1$ symbols, boundedness of the heat transform at one positive time implies boundedness at every positive time, as recalled in~\cite[Section~3]{BauerFulsche2019}. We show in Section~\ref{sec:heat} that a signed symbol satisfying the coherent state condition can have a sharp positive threshold for heat boundedness.

For symbols that are exponentials of complex quadratic polynomials, the Berger--Coburn problem has been studied through Weyl symbols and complex canonical transformations. Under the respective growth and nondegeneracy assumptions, Coburn, Hitrik, and Sj\"ostrand~\cite[Theorem~1.2]{CHS2019} prove that a bounded Weyl symbol gives a bounded Toeplitz operator in the homogeneous quadratic case. Coburn, Hitrik, Sj\"ostrand, and White~\cite[Theorems~1.1 and~4.2]{CHSW2021} extend this sufficiency result to inhomogeneous quadratics and prove necessity for an explicit family. The homogeneous converse was established by Coburn, Hitrik, and Sj\"ostrand~\cite[Theorem~1.1]{CHS2023}, and the inhomogeneous converse by Xiong~\cite[Theorem~1.1]{Xiong2023}.

For the associated operator algebras, Bauer~\cite[Chapter~3]{Bauer2005PsiFrechetThesis} studies smooth $\Psi^\ast$-Fr\'echet subalgebras of Toeplitz $C^\ast$-algebras and estimates involving Berezin transforms. Bauer, Fulsche, and Rodriguez Rodriguez~\cite{BauerFulscheRodriguezRodriguez2024OperatorsFockToeplitzAlgebra} survey and develop membership and compactness questions for several classes of operators in the Fock--Toeplitz algebra. These perspectives complement the direct comparison between the two operator domains considered here.

\paragraph{Proof outline.}
Section~\ref{sec:carleson} proves that quadratic Gaussian averages characterize boundedness of $U_g$, using an adjoint identity and a direct embedding estimate. Section~\ref{sec:counterexample} then gives the smooth radial symbol, proves its spectral and heat decay, and constructs unit vectors on which $U_g$ is unbounded. Section~\ref{sec:heat} uses modulated Gaussians to realize every possible set of positive times at which the heat transform is bounded.

\section{The two embedding conditions}\label{sec:carleson}

The boundedness of $U_g$ is determined by multiplication on the Fock space through an adjoint identity. This reduces the problem to a positive measure, even when $g$ is complex. We give a direct proof of the Gaussian test used below; the underlying Fock--Carleson characterization is classical~\cite{IsralowitzZhu2010Toeplitz}; see also~\cite[Theorem~3.29]{ZhuFock}.

\begin{definition}\label{def:Fock-Carleson measure}
A positive Borel measure $\nu$ on $\C$ is a Carleson measure for $F^2$ if there is a constant $C$ such that
\begin{equation*}
\int_{\C}|h(z)|^2\dd\nu(z)\le C\norm{h}_{F^2}^2
\qquad(h\in F^2)\,.
\end{equation*}
\end{definition}

We use the convention that the Gaussian weight is included in $\nu$. Thus, for a nonnegative density $q$, the measure in this definition is $q\dd\mu$, whereas an equivalent convention puts the factor $e^{-|z|^2}$ in the integrand.

\begin{lemma}\label{lem:kernel-embedding}
Let $q\ge0$ be measurable, and set
\begin{equation*}
\mathcal K(q):=\sup_{a\in\C}\frac1\pi\int_{\C}e^{-|z-a|^2}q(z)\dd A(z)\,.
\end{equation*}
Then $q\dd\mu$ is a Carleson measure for $F^2$ if and only if $\mathcal K(q)<\infty$. If $C_q$ is the least embedding constant, then
\begin{equation}\label{eq:embedding-constants}
\mathcal K(q)\le C_q\le e^2\mathcal K(q)\,.
\end{equation}
\end{lemma}

\begin{proof}
Testing the embedding with $k_a$ gives the first inequality. For the converse, write $Q_h(z)=|h(z)|^2e^{-|z|^2}$. The mean value inequality applied on the unit disk to the entire function $u\to h(z+u)e^{-\overline z u-|z|^2/2}$ gives
\begin{equation}\label{eq:local-mean}
Q_h(z)\le\frac1\pi\int_{|u|<1}Q_h(z+u)e^{|u|^2}\dd A(u)
\le\frac e\pi\int_{B(z,1)}Q_h(w)\dd A(w)\,.
\end{equation}
Also, $\int_{B(w,1)}q(z)\dd A(z)\le\pi e\mathcal K(q)$. Multiplying~\eqref{eq:local-mean} by $q(z)/\pi$ and applying Tonelli's theorem, we obtain
\begin{align*}
\int_{\C}q|h|^2\dd\mu
&\le\frac e{\pi^2}\int_{\C}Q_h(w)
       \left(\int_{B(w,1)}q(z)\dd A(z)\right)\dd A(w)\\
&\le e^2\mathcal K(q)\norm{h}_{F^2}^2\,.\qedhere
\end{align*}
\end{proof}

\begin{theorem}\label{thm:carleson}
Let $g:\C\to\C$ be measurable and finite almost everywhere. The following are equivalent:
\begin{enumerate}[label=(\roman*)]
\item $U_g$ extends to a bounded operator from $L^2(\mu)$ to $F^2$.
\item Multiplication by $g$ is bounded from $F^2$ to $L^2(\mu)$.
\item $|g|^2\dd\mu$ is a Carleson measure for $F^2$.
\item $\displaystyle\sup_{a\in\C}\norm{gk_a}_{L^2(\mu)}^2<\infty$.
\item The integral form $\mathfrak t_{|g|^2}$ is defined and bounded on $F^2\times F^2$.
\end{enumerate}
When these conditions hold,
\begin{equation}\label{eq:adjoint-norm}
U_g^*h=\overline g\,h,\qquad
\norm{U_g}^2=\norm{T_{|g|^2}}\,.
\end{equation}
\end{theorem}

\begin{proof}
Suppose first that $U_g$ has a bounded extension $B$. Fix $h\in F^2$, and let $E_N=\{z:|g(z)|\le N\}$. Every $L^2(\mu)$ function vanishing almost everywhere outside $E_N$ belongs to $\mathcal D(M_g)$, so
\begin{equation*}
\langle f,B^*h\rangle=\langle P(gf),h\rangle
=\langle gf,h\rangle
\qquad\bigl(f=0\text{ a.e. on }E_N^c\bigr)\,.
\end{equation*}
It follows that $B^*h=\overline g\,h$ almost everywhere on $E_N$. Since these sets exhaust $\C$ up to a null set, $\overline g\,h\in L^2(\mu)$ and $\norm{gh}_2\le\norm B\norm h_{F^2}$.

Conversely, if (ii) holds, the map $Ah=\overline g\,h$ is bounded from $F^2$ to $L^2(\mu)$. For $f\in\mathcal D(M_g)$ and $h\in F^2$,
\begin{equation*}
\langle A^*f,h\rangle=\langle f,\overline g\,h\rangle
=\langle gf,h\rangle=\langle U_gf,h\rangle\,.
\end{equation*}
Thus $A^*$ is the bounded extension of $U_g$. This proves (i)$\Leftrightarrow$(ii) and the adjoint formula.

Conditions (ii) and (iii) are the same embedding inequality. Their equivalence with (iv) is Lemma~\ref{lem:kernel-embedding} with $q=|g|^2$. Condition (ii) implies (v) by Cauchy--Schwarz, and (v) implies (ii) by taking the two form arguments equal. Finally, the positive operator $T_{|g|^2}=A^*A$ has norm $\norm A^2=\norm{U_g}^2$.
\end{proof}

For a nonnegative symbol, the same embedding argument applies to $g$ itself. Positivity is essential in the following comparison.

\begin{theorem}[Two Gaussian tests]\label{thm:two-kernel-tests}
Let $g\ge0$ be measurable and finite almost everywhere. Then
\begin{align*}
T_g\text{ is bounded on }F^2
&\quad\Longleftrightarrow\quad
\mathcal K(g)=\norm{g^{(1/4)}}_\infty<\infty\,,\\
U_g\text{ extends boundedly }L^2(\mu)\to F^2
&\quad\Longleftrightarrow\quad
\mathcal K(g^2)<\infty\,.
\end{align*}
For complex $g$, the second equivalence holds with $g^2$ replaced by $|g|^2$.
\end{theorem}

\begin{proof}
If $\mathcal K(g)<\infty$, Lemma~\ref{lem:kernel-embedding} makes the integral form bounded on all of $F^2\times F^2$. Conversely, suppose the densely defined positive form has a bounded extension. Given $f\in F^2$, choose $f_j\in\mathcal D(\mathfrak t_g)$ with $f_j\to f$ in $F^2$. The form bound shows that $\sqrt g\,f_j$ is Cauchy in $L^2(\mu)$. Since multiplication by $\sqrt g$ is closed, its limit is $\sqrt g\,f$, and
\begin{equation*}
\int_{\C}g|f|^2\dd\mu\le\norm{T_g}\norm f_{F^2}^2\,.
\end{equation*}
Testing with $k_a$ proves the first equivalence. The Gaussian integral of a nonnegative function is lower semicontinuous by Fatou's lemma, so its pointwise supremum agrees with its essential supremum. The second equivalence is Theorem~\ref{thm:carleson}.
\end{proof}

\subsection{Local obstructions}

Under the coherent state hypothesis~\eqref{eq:BC-domain}, $g$ belongs to $L^2$ on every compact set, because the condition at $a=0$ gives $g\in L^2(\mu)$. If $g$ is supported in a compact set, this condition is equivalent to $g\in L^2(\dd A)$ and hence to boundedness of $U_g$: indeed,
\begin{equation*}
\mathcal K(|g|^2)\le\frac1\pi\int_{\C}|g|^2\dd A\,.
\end{equation*}
More generally, for every $R<\infty$, the map $U_{g\mathbf1_{B(0,R)}}$ is bounded. Thus the failure in Theorem~\ref{thm:main} requires contributions from arbitrarily large radii.

\begin{remark}\label{rem:cgt}
The coherent state hypothesis puts a dense family of holomorphic functions in $\mathcal D(M_g)$, but it need not put all of $F^2$ there. In fact,
\begin{equation*}
F^2\subset\mathcal D(M_g)
\quad\Longleftrightarrow\quad
U_g\text{ has a bounded extension }L^2(\mu)\to F^2\,.
\end{equation*}
For the forward implication, $M_g|_{F^2}$ is an everywhere defined closed operator, so the closed graph theorem makes it bounded; Theorem~\ref{thm:carleson} then applies. The reverse implication is part (ii) of that theorem. To check closedness, if $f_j\to f$ in $F^2$ and $gf_j\to v$ in $L^2(\mu)$, pointwise convergence of $f_j$ and an almost everywhere convergent subsequence of $gf_j$ give $v=gf$.
\end{remark}

Power symbols illustrate the local integrability thresholds. Their separation range is excluded by the coherent state hypothesis, which explains the need for a different construction in Theorem~\ref{thm:main}.

\begin{proposition}\label{prop:powers}
For $g(z)=|z|^\alpha$, with $\alpha\in\mathbb R$, the following hold:
\begin{enumerate}[label=(\roman*)]
\item For each $t>0$, $g^{(t)}\in L^\infty(\C)$ if and only if $-2<\alpha\le0$.
\item $g$ defines a bounded Toeplitz operator $T_g$ if and only if $-2<\alpha\le0$.
\item $U_g$ has a bounded extension if and only if $-1<\alpha\le0$.
\end{enumerate}
\end{proposition}

\begin{proof}
Adding a bounded symbol changes neither boundedness question. Indeed, if $b\in L^\infty(\C)$, then $\mathcal D(M_{g+b})=\mathcal D(M_g)$ and $U_{g+b}=U_g+U_b$ on this common domain, with $U_b$ bounded. The inequalities $|g+b|\le|g|+\norm b_\infty$ and $|g|\le|g+b|+\norm b_\infty$ also give $\mathcal D(\mathfrak t_{g+b})=\mathcal D(\mathfrak t_g)$. On this form domain, $\mathfrak t_{g+b}=\mathfrak t_g+\mathfrak t_b$, and $\mathfrak t_b$ is bounded on $F^2\times F^2$. Hence $T_{g+b}$ is bounded if and only if $T_g$ is bounded, with the same requirement that the domain be dense.

For $\alpha>0$, restricting a Gaussian average to $B(a,1)$ gives a lower bound $c_t(|a|-1)^\alpha$ when $|a|>1$. This proves unboundedness of every heat transform. Applying the same argument to $g$ and $g^2$, the two Gaussian tests rule out boundedness in (ii) and (iii).

For $\alpha\le0$, the part of $g$ outside the unit disk is bounded. On the unit disk,
\begin{equation*}
\int_{|z|<1}|z|^\alpha\dd A(z)=2\pi\int_0^1r^{\alpha+1}\dd r<\infty
\quad\Longleftrightarrow\quad\alpha>-2\,.
\end{equation*}
The corresponding threshold for $g^2$ is $\alpha>-1$. A nonnegative compactly supported density $q$ has $\mathcal K(q)<\infty$ exactly when $q\in L^1(\dd A)$: sufficiency follows by bounding the Gaussian by one, and necessity follows by testing at the origin. The two Gaussian tests therefore prove (ii) and (iii).

For (i), the integrable part on the disk has heat transform bounded by $(4\pi t)^{-1}\norm{g\mathbf1_{B(0,1)}}_1$, and the bounded part outside the disk has a bounded heat transform. If $\alpha\le-2$, the heat integral diverges at every center because of the singularity at the origin.
\end{proof}

When $-2<\alpha\le-1$, the form defines a bounded $T_g$ but $U_g$ is unbounded, and already $gk_0\notin L^2(\mu)$. For $\alpha\le-2$, the integral form is not densely defined: every function in its domain must vanish at the origin. The convention in the introduction requiring a dense form domain rules out interpreting a nonunique bounded extension from such a proper closed subspace as $T_g$.

\section{A smooth radial construction}\label{sec:counterexample}

We now prove Theorem~\ref{thm:main} by choosing the annular widths so that $g(z)e^{A|z|^2}$ is integrable for every $A>0$, while the quadratic mass on a small sector of each annulus tends to infinity. From the weighted integrability we obtain Schwartz heat transforms and, using radiality, rapid decay of the Toeplitz eigenvalues. The growing quadratic mass makes $U_g$ unbounded, and our choice of widths also gives $g\in L^p(\dd A)$ for every $0<p<2$.

For boundedness of the Toeplitz form, unweighted $L^1$ integrability suffices. We use the Gaussian weights and radiality when estimating the eigenvalues.

\begin{samepage}
\begin{proposition}\label{prop:Tg}
If $g\in L^1(\dd A)$, then $\mathcal D(\mathfrak t_g)=F^2$, and
\begin{equation}\label{eq:L1-form}
\norm{T_g}\le\frac1\pi\int_{\C}|g|\dd A\,.
\end{equation}
If also $g\ge0$, then $T_g$ is positive and trace class, with
\begin{equation}\label{eq:trace}
\tr T_g=\frac1\pi\int_{\C}g\dd A\,.
\end{equation}
\end{proposition}
\end{samepage}

\begin{proof}
The reproducing kernel estimate $|f(z)|e^{-|z|^2/2}\le\norm f_{F^2}$ gives
\begin{equation*}
\int_{\C}|g f\overline h|\dd\mu
\le\frac1\pi\norm f_{F^2}\norm h_{F^2}\int_{\C}|g|\dd A\,.
\end{equation*}
This proves the domain assertion and~\eqref{eq:L1-form}. If $g\ge0$, use the orthonormal basis $e_m(z)=z^m/\sqrt{m!}$ and Tonelli's theorem to obtain
\begin{equation*}
\sum_{m=0}^{\infty}\langle T_ge_m,e_m\rangle
=\frac1\pi\int_{\C}g(z)e^{-|z|^2}
       \sum_{m=0}^{\infty}\frac{|z|^{2m}}{m!}\dd A(z)
=\frac1\pi\int_{\C}g\dd A<\infty\,.
\end{equation*}
The positive operator $T_g$ is therefore trace class, with the stated trace.
\end{proof}

\begin{proof}[Proof of Theorem~\ref{thm:main}]
Choose $\psi\in C_c^\infty((-1,1))$ with $0\le\psi\le1$ and $\psi=1$ on $[-1/2,1/2]$. For $n\ge1$, put
\begin{equation}\label{eq:params}
a_n=n,\qquad \rho_n=e^{-2n^4},\qquad d_n=ne^{n^4}\,,
\end{equation}
and define
\begin{equation}\label{eq:g}
g(z)=\sum_{n=1}^{\infty}d_n\psi\left(\frac{|z|-a_n}{\rho_n}\right)\,.
\end{equation}
The supports lie in the annuli
\begin{equation*}
A_n=\{z:\bigl||z|-a_n\bigr|\le\rho_n\}\,.
\end{equation*}
These annuli are disjoint, since
\begin{equation*}
\rho_n+\rho_{n+1}\le2e^{-2}<1=a_{n+1}-a_n\,.
\end{equation*}
The sum is locally finite and vanishes near the origin, so $g$ is smooth, radial, and nonnegative.

\emph{Step 1: Linear mass, heat decay, and singular values.}
The area identity $|A_n|=4\pi a_n\rho_n=4\pi ne^{-2n^4}$ gives, for every $0<p<2$,
\begin{equation}\label{lem:L1}
\int_{\C}g^p\dd A\le\sum_{n=1}^{\infty}d_n^p|A_n|
=4\pi\sum_{n=1}^{\infty}n^{p+1}e^{-(2-p)n^4}<\infty\,.
\end{equation}
Taking $p=1$, we obtain, for every $t>0$,
\begin{equation}\label{eq:heat_transform_estimate}
0\le g^{(t)}(a)\le\frac1{4\pi t}\norm g_{L^1(\dd A)}
\qquad(a\in\C)\,.
\end{equation}
Proposition~\ref{prop:Tg} also makes $T_g$ positive and trace class.

The thinner annuli give more than trace class. For every $A>0$,
\begin{equation}\label{eq:weighted-linear-mass}
M_A:=\frac1\pi\int_{\C}g(z)e^{A|z|^2}\dd A(z)
\le4\sum_{n=1}^{\infty}n^2e^{-n^4+A(n+1)^2}<\infty\,.
\end{equation}
Completing the square in the heat kernel therefore gives
\begin{equation}\label{eq:heat-Gaussian-decay}
0\le g^{(t)}(a)\le\frac{M_A}{4t}
          e^{-A|a|^2/(1+4At)}\qquad(A,t>0)\,.
\end{equation}
Every real partial derivative of the heat kernel at time $t$ is bounded in absolute value by a constant times the heat kernel at time $2t$. Since $g\in L^1(\dd A)$, differentiation under the integral is valid, and each derivative of $g^{(t)}$ is bounded by a constant times $g^{(2t)}$. Estimate~\eqref{eq:heat-Gaussian-decay} thus gives $g^{(t)}\in\mathcal S(\C)$ for every $t>0$.

Since $g$ is radial, angular integration shows that $T_g$ is diagonal in the monomial basis, with eigenvalues
\begin{equation*}
\lambda_m=\frac1{\pi m!}\int_{\C}g(z)|z|^{2m}e^{-|z|^2}\dd A(z)
\qquad(m\ge0)\,.
\end{equation*}
The inequality $u^m e^{-u}/m!\le1$ for $u\ge0$, applied with $u=(1+A)|z|^2$, gives $\lambda_m\le M_A(1+A)^{-m}$. Truncating this diagonal operator after its first $j-1$ monomials therefore gives
\begin{equation*}
s_j(T_g)\le\sup_{m\ge j-1}\lambda_m
\le M_A(1+A)^{-(j-1)}\,.
\end{equation*}
Taking $A=e^c-1$ proves~\eqref{eq:singular-decay}, and summation gives $T_g\in S_q$ for every $q>0$.

\emph{Step 2: Coherent state admissibility.}
Fix $a\in\C$. Since the supports are disjoint,
\begin{equation*}
\norm{gk_a}_2^2\le\frac1\pi\sum_{n=1}^{\infty}d_n^2
       \int_{A_n}e^{-|z-a|^2}\dd A(z)\,.
\end{equation*}
For all sufficiently large $n$, depending on $a$, we have $|z-a|\ge n/2$ on $A_n$. The tail of this sum is bounded by
\begin{equation}\label{lem:annuli-coherent-state}
\frac1\pi\sum_{n\ge N(a)}d_n^2|A_n|e^{-n^2/4}
=4\sum_{n\ge N(a)}n^3e^{-n^2/4}<\infty\,.
\end{equation}
Every remaining term is finite, proving $gk_a\in L^2(\mu)$.

\emph{Step 3: Unboundedness of projected multiplication.}
Choose a sector inside the plateau of the $n$th bump:
\begin{equation}\label{eq:En}
\phi_n=\frac1{4n},\qquad
E_n=\left\{re^{i\theta}:|r-a_n|\le\frac{\rho_n}2,
                      \ |\theta|\le\phi_n\right\}\,.
\end{equation}
On $E_n$ we have $g=d_n$, and direct integration gives
\begin{equation}\label{lem:En-area}
|E_n|=2a_n\rho_n\phi_n=\frac{\rho_n}2\,.
\end{equation}
Regarding $a_n$ as a point on the positive real axis, for $z\in E_n$ we have
\begin{equation*}
|z-a_n|\le\frac{\rho_n}2+\left(a_n+\frac{\rho_n}2\right)\phi_n
\le\frac78<1\,.
\end{equation*}
Thus, with $q_n=\int_{E_n}|k_{a_n}|^2\dd\mu$,
\begin{equation}\label{eq:sector-mass}
q_n=\frac1\pi\int_{E_n}e^{-|z-a_n|^2}\dd A(z)
\ge\frac{\rho_n}{2\pi e}\,.
\end{equation}
Define $f_n=q_n^{-1/2}\mathbf1_{E_n}k_{a_n}$. Then $\norm{f_n}_2=1$ and $gf_n=d_nf_n\in L^2(\mu)$, so $f_n\in\mathcal D(M_g)$. Pairing $U_gf_n$ with the unit vector $k_{a_n}$ yields
\begin{equation}\label{eq:Ug-lower}
\norm{U_gf_n}_{F^2}^2
\ge|\langle U_gf_n,k_{a_n}\rangle|^2
=d_n^2q_n
\ge\frac{n^2}{2\pi e}\to\infty\,.
\end{equation}
Hence $U_g$ has no bounded extension.
\end{proof}

The same sectors give the pointwise versus uniform distinction directly:
\begin{equation}\label{eq:kernel-divergence}
\norm{gk_{a_n}}_2^2\ge d_n^2q_n\to\infty\,,
\end{equation}
although Step~2 proves finiteness at every fixed center. This proves Corollary~\ref{cor:global}. We only used the region where the bump has full height, so smoothing its edges does not change this lower bound.

We can also take the test vectors in $C_c^\infty(\C)$. Indeed, take $0\le\eta_n\le1$ smooth, supported in the interior of $E_n$, and equal to one on the sector with both widths halved. Set $v_n=\eta_nk_{a_n}/\norm{\eta_nk_{a_n}}_2$. Since $\eta_n\ge\eta_n^2$ and $g=d_n$ on its support, pairing with $k_{a_n}$ gives
\begin{equation*}
\norm{U_gv_n}_2^2\ge d_n^2\norm{\eta_nk_{a_n}}_2^2\ge c n^2\,,
\end{equation*}
where $c>0$ is independent of $n$. Thus the failure of boundedness already occurs on smooth functions with compact support.

\subsection{A symbol with polynomial growth}\label{sec:polynomial-growth}

The heights in~\eqref{eq:params} grow faster than every polynomial in the radius. We now show that the same operator separation is possible with a symbol of polynomial growth. The Schwartz heat conclusion must then be weakened: as we show after the proposition, polynomial growth and nonnegativity make a Schwartz heat transform at time $1/2$ sufficient for boundedness of $U_g$. Using polynomial annular scales, we obtain heat transforms in $C_0(\C)$ for every positive time while keeping $U_g$ unbounded.

\begin{proposition}\label{prop:polynomial-growth}
There is a smooth, nonnegative radial symbol $g\in L^1(\dd A)$ satisfying
\begin{equation}\label{eq:polynomial-growth}
g(z)\le C(1+|z|)^5\log(2+|z|)\qquad(z\in\C)
\end{equation}
such that $gk_a\in L^2(\mu)$ for every $a\in\C$, $T_g$ is positive and trace class, and $g^{(t)}\in C_0(\C)\setminus\mathcal S(\C)$ for every $t>0$, but $U_g$ has no bounded extension.
\end{proposition}

\begin{proof}
Take a bump $\psi$ as in the proof of Theorem~\ref{thm:main}, and set
\begin{equation*}
a_n=\sqrt n,\qquad \rho_n=n^{-9/2},\qquad d_n=n^{5/2}\log n
\qquad(n\ge2)\,.
\end{equation*}
Define $g(z)=\sum_{n\ge2}d_n\psi((|z|-a_n)/\rho_n)$. The annuli are disjoint because $\rho_n+\rho_{n+1}\le2n^{-9/2}<\sqrt{n+1}-\sqrt n$ for $n\ge2$. Thus $g$ is smooth, radial, and nonnegative. On the $n$th annulus, $|z|$ is comparable to $\sqrt n$, which gives~\eqref{eq:polynomial-growth}. This growth bound also makes every Gaussian average of $g^2$ finite, so $g$ satisfies~\eqref{eq:BC-domain}. Moreover,
\begin{equation*}
\int_{\C}g\dd A\le4\pi\sum_{n\ge2}d_na_n\rho_n
=4\pi\sum_{n\ge2}\frac{\log n}{n^{3/2}}<\infty\,.
\end{equation*}
Proposition~\ref{prop:Tg} makes $T_g$ positive and trace class. Convolution with the heat kernel maps $L^1(\dd A)$ into $C_0(\C)$, as follows by approximation in $L^1$ by functions with compact support, so $g^{(t)}\in C_0(\C)$ for every $t>0$.

For the quadratic lower bound, let
\begin{equation*}
E_n=\left\{re^{i\theta}:|r-a_n|\le\frac{\rho_n}2,
                   \ |\theta|\le\frac1{4n}\right\}\,.
\end{equation*}
Then $g=d_n$ on $E_n$, $|E_n|=\rho_n/(2\sqrt n)=1/(2n^5)$, and $|z-a_n|\le\rho_n/2+(a_n+\rho_n/2)/(4n)<1$ for $z\in E_n$. Consequently,
\begin{equation*}
\norm{gk_{a_n}}_2^2
\ge\frac{e^{-1}}\pi d_n^2|E_n|
=\frac{(\log n)^2}{2\pi e}\to\infty\,.
\end{equation*}
Theorem~\ref{thm:carleson} now gives the unboundedness of $U_g$.

The same sectors give, for every fixed $t>0$,
\begin{equation*}
g^{(t)}(a_n)\ge\frac{e^{-1/(4t)}}{4\pi t}d_n|E_n|
=\frac{e^{-1/(4t)}}{8\pi t}\frac{\log n}{n^{5/2}}\,.
\end{equation*}
Thus $a_n^5g^{(t)}(a_n)\to\infty$, so no heat transform at a positive time is a Schwartz function.
\end{proof}

We now verify the obstruction for nonnegative symbols of polynomial growth. Suppose $0\le g(z)\le C(1+|z|)^N$ for some $N\ge0$. The inequalities $g^2\le C(1+|z|)^Ng$ and $1+|z|\le(1+|a|)(1+|z-a|)$ give
\begin{align*}
\norm{gk_a}_2^2
&\le C_N(1+|a|)^N\frac1\pi
  \int_{\C}e^{-|z-a|^2/2}g(z)\dd A(z)\\
&=2C_N(1+|a|)^Ng^{(1/2)}(a)\,.
\end{align*}
Here $C_N$ depends only on $C$ and $N$, since the remaining polynomial is absorbed by the Gaussian. If $g^{(1/2)}$ is a Schwartz function, this estimate is uniform in $a$, and Theorem~\ref{thm:carleson} makes $U_g$ bounded.

\section{An exact threshold for heat boundedness}\label{sec:heat}

The preceding constructions have bounded heat transforms at every positive time. For nonnegative symbols, this is already forced by boundedness at a single positive time; we first recall the quantitative estimate. We then allow cancellation and prove that smooth real symbols satisfying the coherent state condition can realize every set of heat boundedness times permitted by the semigroup property.

\begin{lemma}\label{lem:positive-heat}
Let $h\ge0$ be measurable. If $h^{(t_0)}\in L^\infty(\C)$ for some $t_0>0$, then, for every $t>0$,
\begin{equation}\label{eq:positive-heat}
\norm{h^{(t)}}_\infty
\le\max\{1,t_0/t\}\norm{h^{(t_0)}}_\infty\,.
\end{equation}
\end{lemma}

\begin{proof}
For $0<t<t_0$, the heat kernels satisfy
\begin{equation*}
\frac1{4\pi t}e^{-|u|^2/(4t)}
\le\frac{t_0}t\frac1{4\pi t_0}e^{-|u|^2/(4t_0)}\,.
\end{equation*}
Integrating against $h\ge0$ gives the bound. For $t>t_0$, Tonelli's theorem and the Gaussian convolution identity give $h^{(t)}=(h^{(t_0)})^{(t-t_0)}$, and the heat operator contracts the $L^\infty$ norm.
\end{proof}

Combining Lemma~\ref{lem:positive-heat} with the positive Gaussian test gives the classical positive Toeplitz criterion in our heat normalization; compare~\cite[Theorems~3.1 and~4.2]{HuLv2011}. We include the short proof to connect the operator comparison with the question about heat times.

\begin{corollary}\label{cor:positive-heat-criterion}
Let $g\ge0$ be measurable and finite almost everywhere. The following are equivalent:
\begin{enumerate}[label=(\roman*)]
\item $T_g$ is bounded on $F^2$.
\item $g^{(t)}\in L^\infty(\C)$ for some $t>0$.
\item $g^{(t)}\in L^\infty(\C)$ for every $t>0$.
\end{enumerate}
\end{corollary}

\begin{proof}
Theorem~\ref{thm:two-kernel-tests} identifies (i) with boundedness of $g^{(1/4)}$. Lemma~\ref{lem:positive-heat} makes boundedness at this time equivalent to (ii) and (iii).
\end{proof}

For signed symbols with absolutely convergent heat integrals at every positive time, the semigroup property restricts the set of boundedness times to an upper interval. We prove that every such set occurs for a smooth real symbol satisfying the coherent state condition. For any $t_0>0$, we can choose whether the heat transform is bounded at $t_0$, while keeping it unbounded for $0<t<t_0$ and Schwartz for every $t>t_0$.

\begin{theorem}\label{thm:heat-irreversibility}
Let $g:\C\to\C$ be measurable, and suppose that every heat integral $g^{(t)}(a)$ is absolutely convergent for $a\in\C$ and $t>0$. Then
\begin{equation*}
\mathcal B(g):=\{t>0:g^{(t)}\in L^\infty(\C)\}
\end{equation*}
is one of the sets
\begin{equation*}
\varnothing,\qquad (0,\infty),\qquad [t_0,\infty),\qquad (t_0,\infty)
\qquad(t_0>0)\,.
\end{equation*}
Every set in this list is realized by a smooth real symbol satisfying $gk_a\in L^2(\mu)$ for every $a\in\C$, with every heat integral at a positive time absolutely convergent. For either interval with $t_0>0$, the symbol can also be chosen so that $g^{(t)}\in\mathcal S(\C)$ for every $t>t_0$.
\end{theorem}

\begin{proof}
Write $K_t(z)=(4\pi t)^{-1}e^{-|z|^2/(4t)}$. For $r,s>0$, Tonelli's theorem and Gaussian convolution give
\begin{equation*}
\int_{\C}\int_{\C}K_r(z-u)K_s(u-w)|g(w)|\dd A(w)\dd A(u)
=|g|^{(r+s)}(z)<\infty\,.
\end{equation*}
Thus Fubini's theorem gives $g^{(r+s)}=(g^{(s)})^{(r)}$. In particular, if $s\in\mathcal B(g)$, then $r+s\in\mathcal B(g)$ and $\norm{g^{(r+s)}}_\infty\le\norm{g^{(s)}}_\infty$. Hence $\mathcal B(g)$ is an upper interval of one of the stated forms. The symbols $g=0$ and $g(z)=|z|^2$, for which $g^{(t)}(a)=|a|^2+4t$ in the latter case, realize $(0,\infty)$ and $\varnothing$, respectively; both satisfy the required domain and convergence conditions.

Fix $t_0>0$. For the two remaining intervals, we choose the amplitudes of modulated Gaussians to offset the exponential damping at $t_0$, and vary a polynomial factor to include or exclude the threshold. We place the centers on the real axis, so that every cosine factor equals one there, and let the centers grow quadratically with the index to make the series converge.

For $n\ge1$, put $h_n(z)=e^{-|z|^2}\cos(n\Ima z)$, and write
\begin{equation*}
\beta(t)=\frac1{1+4t},\qquad \alpha(t)=\frac t{1+4t}\,.
\end{equation*}
Completing the square in the two real coordinates gives
\begin{equation}\label{eq:modulated-heat}
h_n^{(t)}(z)=\beta(t)e^{-\beta(t)|z|^2-\alpha(t)n^2}
                     \cos\bigl(\beta(t)n\Ima z\bigr)\,.
\end{equation}
The function $\alpha$ is strictly increasing. Set $\alpha_0=\alpha(t_0)$, choose $\varepsilon\in\{0,1\}$, and define
\begin{equation}\label{eq:heat-series}
b_n=n^\varepsilon e^{\alpha_0 n^2},\qquad z_n=10n^2,\qquad
g(z)=\sum_{n=1}^{\infty}b_nh_n(z-z_n)\,.
\end{equation}
For each multiindex $\nu$ in the real coordinates and $|z|\le R$, the derivatives of the summands satisfy, for all sufficiently large $n$,
\begin{equation*}
\bigl|D^\nu[b_nh_n(z-z_n)]\bigr|
\le C_{R,\nu}n^{\varepsilon+2|\nu|}
           e^{\alpha_0n^2-(10n^2-R)^2}\,.
\end{equation*}
The quartic decay makes every derivative series locally uniformly convergent. Thus $g$ is smooth and real.

\emph{Step 1: Absolute convergence of the heat integrals.}
For each fixed $t>0$, Tonelli's theorem and the Gaussian convolution formula give
\begin{align*}
\frac1{4\pi t}\int_{\C}\sum_{n=1}^{\infty}b_ne^{-|w-z_n|^2}
                          e^{-|z-w|^2/(4t)}\dd A(w)
&=\beta(t)\sum_{n=1}^{\infty}b_ne^{-\beta(t)|z-z_n|^2}\\
&<\infty\,.
\end{align*}
The series on the right converges locally uniformly, since the negative term in its exponent has order $n^4$, whereas $\log b_n$ has order $n^2$. Thus every heat integral is absolutely convergent, termwise convolution is valid, and $g^{(t)}$ is continuous.

\emph{Step 2: Boundedness and decay after the threshold.}
Put $c_n(t)=n^\varepsilon e^{(\alpha_0-\alpha(t))n^2}$. For $t>t_0$, these coefficients are bounded and tend to zero as $n\to\infty$. Formula~\eqref{eq:modulated-heat} gives
\begin{equation*}
|g^{(t)}(z)|\le\beta(t)\sup_nc_n(t)
                  \sum_{n=1}^{\infty}e^{-\beta(t)|z-z_n|^2}\,.
\end{equation*}
For each fixed $t>0$, this sum is uniformly bounded in $z$. Indeed, the centers are separated by at least $10$, so disjoint disks about them show that the number of centers with $k\le|z-z_n|<k+1$ is at most $C(k+1)^2$, uniformly in $z$. Summing $C(k+1)^2e^{-\beta(t)k^2}$ over $k\ge0$ gives the assertion.

In fact, fix $t>t_0$ and put $\delta=\alpha(t)-\alpha_0>0$. Differentiating~\eqref{eq:modulated-heat}, and using $1+|z|\le(1+|z_n|)(1+|z-z_n|)$, gives, for every integer $N\ge0$ and multiindex $\nu$ in the real coordinates,
\begin{equation*}
\sup_{z\in\C}(1+|z|)^N
   \bigl|D^\nu[b_nh_n^{(t)}(z-z_n)]\bigr|
\le C_{N,\nu,t}n^{\varepsilon+2N+|\nu|}e^{-\delta n^2}\,.
\end{equation*}
The right side is summable, so the heat series converges in every Schwartz seminorm. Hence $g^{(t)}\in\mathcal S(\C)$ for $t>t_0$. At $t=t_0$, the bound on the Gaussian sum proves boundedness when $\varepsilon=0$, since $c_n(t_0)=1$.

\emph{Step 3: The remaining times.}
Since every $z_n-z_m$ is real, all cosine factors in the heat series equal one at $z_n$. Every summand there is positive, and hence
\begin{equation*}
g^{(t)}(z_n)\ge\beta(t)n^\varepsilon
                    e^{(\alpha_0-\alpha(t))n^2}\,.
\end{equation*}
For $0<t<t_0$, this tends to infinity for both choices of $\varepsilon$. At $t=t_0$, it tends to infinity when $\varepsilon=1$. Continuity of $g^{(t)}$ gives essential unboundedness. Thus $\varepsilon=0$ realizes $[t_0,\infty)$, while $\varepsilon=1$ realizes $(t_0,\infty)$. In the closed case, $g^{(t_0)}(z_n)\ge\beta(t_0)$ also shows that the bounded endpoint transform does not belong to $C_0(\C)$.

\emph{Step 4: Coherent state admissibility.}
Separation gives $\sup_z\sum_ne^{-|z-z_n|^2}<\infty$, so Cauchy--Schwarz implies
\begin{equation*}
|g(z)|^2\le\left(\sum_ne^{-|z-z_n|^2}\right)
                  \left(\sum_nb_n^2e^{-|z-z_n|^2}\right)
\le C\sum_nb_n^2e^{-|z-z_n|^2}\,.
\end{equation*}
For each fixed $a\in\C$, completing the square gives
\begin{align*}
\norm{gk_a}_2^2
&\le\frac C\pi\sum_nb_n^2\int_{\C}
                    e^{-|z-z_n|^2-|z-a|^2}\dd A(z)\\
&=\frac C2\sum_nb_n^2e^{-|z_n-a|^2/2}\\
&\le\frac C2 e^{|a|^2/2}
                  \sum_{n=1}^{\infty}n^{2\varepsilon}e^{2\alpha_0n^2-25n^4}<\infty\,.
\end{align*}
Here we used $|z_n-a|^2\ge|z_n|^2/2-|a|^2$.
\end{proof}

Under~\eqref{eq:BC-domain}, the Berger--Coburn sufficiency estimate gives boundedness of the Toeplitz compression if $g^{(t)}$ is bounded for some $0<t<1/8$; see~\cite[Theorem~3.1]{BauerFulsche2019}. Here the Fock parameter in that theorem is $1$, and its regularization parameter is four times our heat time.

Our construction with $t_0=1/8$ and the endpoint included gives a single symbol satisfying~\eqref{eq:BC-domain} for which $g^{(1/8)}$ is bounded and $g^{(1/8-\varepsilon)}$ is unbounded for every $0<\varepsilon<1/8$. Thus the critical heat bound alone cannot supply an earlier bounded heat transform to which the sufficiency estimate applies. Lemma~\ref{lem:positive-heat} explains why such an example must use cancellation.

\section*{Acknowledgements}
The author thanks Otte Hein\"avaara, Michael Hitrik, Jeck Lim, Izak Oltman, and Jared Wunsch for valuable discussions. This work grew out of the American Institute of Mathematics workshop on Riemann--Hilbert problems, Toeplitz matrices, and applications. The author is an Olga Taussky and John Todd Fellow at Caltech and was supported by the PMA Division of Caltech and NSF grant DMS-2346799. The author gratefully acknowledges the hospitality of the American Institute of Mathematics and the Simons Laufer Mathematical Sciences Institute (SLMath), where part of this work was completed.

\end{document}